\documentclass[a4paper]{llncs}
\usepackage{fourier}
\usepackage{amsmath,amsfonts,amssymb}

\title{
    Families of fast elliptic curves 
    from \(\QQ\)-curves 
}
\author{Benjamin Smith}
\institute{
    Team GRACE, INRIA Saclay--\^Ile-de-France
    \\
    \emph{and}
    Laboratoire d'Informatique de l'\'Ecole polytechnique
    (LIX)
    \\
    B\^atiment Alan Turing, 1 rue Honor\'e d'Estienne d'Orves
    \\
    Campus de l'\'Ecole polytechnique, 
    91120 Palaiseau, France
}
\date{\today}

\numberwithin{equation}{section}

\newcommand{\ZZ}{\mathbb{Z}}
\newcommand{\QQ}{\mathbb{Q}}
\newcommand{\QQbar}{\overline{\mathbb{Q}}}
\newcommand{\Kbar}{\overline{K}}

\newcommand{\FF}{\mathbb{F}}
\newcommand{\FFbar}{\overline{\mathbb{F}}}

\newcommand{\EC}{\mathcal{E}}
\newcommand{\ECK}{\widetilde{\mathcal{E}}}
\newcommand{\phiK}{\widetilde{\phi}}
\newcommand{\G}{\mathcal{G}}

\newcommand{\Lattice}{\mathcal{L}}
\newcommand{\End}{\mathrm{End}}

\newcommand{\Gal}{\mathrm{Gal}}

\newcommand{\trace}[1]{{\mathrm{tr}({#1})}}

\newcommand{\subgrp}[1]{\langle{#1}\rangle}
\newcommand{\roundoff}[1]{\lfloor{#1}\rceil}
\newcommand{\bigroundoff}[1]{\big\lfloor{#1}\big\rceil}

\newcommand{\conj}[2][\sigma]{{{}^{#1}\!\!{#2}}}
\newcommand{\dualof}[1]{{{#1}^{\dagger}}}
\newcommand{\twist}[2]{#2^{#1}}
\newcommand{\twistiso}[1]{\delta({#1})}

\newcommand{\Legendre}[2]{\left({#1}\big/{#2}\right)}

\newcommand{\Oh}{O}

\renewcommand{\vec}[1]{\mathbf{#1}}

\begin{document}
\maketitle

\begin{abstract}
    We construct new families of elliptic curves over \(\FF_{p^2}\)
    with efficiently computable endomorphisms,
    which can be used
    to accelerate elliptic curve-based cryptosystems
    in the same way as Gallant--Lambert--Vanstone (GLV) 
    and Galbraith--Lin--Scott (GLS) endomorphisms.
    Our construction is based on reducing \(\QQ\)-curves---curves over 
    quadratic number fields without complex multiplication,
    but with isogenies to their Galois conjugates---modulo
    inert primes.  
    As a first application of the general theory we construct, for every \(p > 3\),
    two one-parameter families of elliptic curves over
    \(\FF_{p^2}\) equipped with endomorphisms
    that are faster than doubling.
    Like GLS (which appears as a degenerate case of our construction), 
    we offer the advantage over GLV of selecting from 
    a much wider range of curves, and thus finding secure group orders 
    when \(p\) is fixed.
    Unlike GLS, we also offer the possibility of constructing
    twist-secure curves.
    Among our examples are prime-order curves equipped with fast
    endomorphisms, with almost-prime-order twists, over \(\FF_{p^2}\)
    for \(p = 2^{127}-1\) and \(p = 2^{255}-19\).
\end{abstract}


\section{
    Introduction
}

Let \(\EC\) be an elliptic curve over a finite field \(\FF_{q}\),
and let \(\G \subset \EC(\FF_{q})\)
be a cyclic subgroup of prime order \(N\).
When implementing cryptographic protocols using \(\G\),
the fundamental operation is
\emph{scalar multiplication}
(or \emph{exponentiation}):
\begin{center}
    Given \(P\) in \(\G\) and \(m\) in \(\ZZ\),
    compute 
    \([m]P := \underbrace{P\oplus\cdots\oplus P}_{m \text{ times}} \).
\end{center}

The literature on general scalar multiplication algorithms is vast, 
and we will not explore it in detail here 
(see~\cite[\S2.8,\S11.2]{Galbraith} 
and~\cite[Chapter 9]{Handbook}
for introductions to exponentiation and 
multiexponentiation algorithms).
For our purposes,
it suffices to note that
the dominant factor in scalar multiplication time
using conventional algorithms
is the bitlength of \(m\).
As a basic example,
if \(\G\) is a generic cyclic abelian group,
then we may
compute \([m]P\)
using a variant of the binary method,
which requires at most \(\lceil\log_2 m\rceil\) doublings
and (in the worst case) about as many addings
in~\(\G\).

But elliptic curves are not generic groups:
they have a rich and concrete geometric structure,
which should be exploited for fun and profit.
For example,
endomorphisms of elliptic curves may be used to accelerate
generic scalar multiplication algorithms,
and thus to accelerate basic operations in curve-based cryptosystems.

Suppose \(\EC\) is 
equipped with an efficient
endomorphism~\(\psi\), defined over \(\FF_{q}\).
By \emph{efficient},
we mean that we can compute the image \(\psi(P)\)
of any point \(P\) in \(\EC(\FF_{q})\) 
for the cost of \(O(1)\) operations in~\(\FF_{q}\).
In practice, we want this to cost no more than a few doublings in
\(\EC(\FF_{q})\).

Assume \(\psi(\G) \subseteq \G\),
or equivalently,
that \(\psi\) restricts to an endomorphism of \(\G\).\footnote{
    This assumption is satisfied almost by default
    in the context of old-school discrete log-based
    cryptosystems.
    If \(\psi(\G) \not\subseteq \G\),
    then \(\EC[N](\FF_{q}) = \G+\psi(\G) \cong (\ZZ/N\ZZ)^2\),
    so \(N^2\mid\#\EC(\FF_{q})\)
    and \(N\mid q-1\);
    such \(\EC\) are cryptographically inefficient,
    and discrete logs in \(\G\) are vulnerable to the
    Menezes--Okamoto--Vanstone reduction~\cite{MOV}.
    However, these \(\G\) do arise naturally in pairing-based
    cryptography; in that context the assumption
    should be verified carefully.
}
Now \(\G\) is a finite cyclic group, isomorphic to \(\ZZ/N\ZZ\);
and every endomorphism of \(\ZZ/N\ZZ\) is just an integer
multiplication modulo~\(N\).
Hence, \(\psi\) acts on \(\G\) as multiplication by some integer eigenvalue
\(\lambda_\psi\): that is,
\[
    \psi|_\G = [\lambda_\psi]_\G 
    .
\]
The eigenvalue \(\lambda_\psi\) must be a root in \(\ZZ/N\ZZ\)
of the characteristic polynomial of \(\psi\).

Returning to the problem of scalar multiplication:
we want to compute \([m]P\).
Rewriting \(m\) as
\[
    m = a + b\lambda_\psi \pmod N
\]
for some \(a\) and \(b\),
we can compute \([m]P\) using the relation
\[
    [m]P = [a]P + [b\lambda_\psi]P = [a]P + [b]\psi(P) 
\]
and a two-dimensional multiexponentation
such as Straus's algorithm~\cite{Straus},
which requires has a loop length of \(\log_2\|(a,b)\|_\infty\)
(ie, \(\log_2\|(a,b)\|_\infty\) doubles and as many adds;
recall that \(\|(a,b)\|_\infty = \max(|a|,|b|)\)).
If \(\lambda_\psi\) is not too small,
then we can easily find \((a,b)\)
such that \(\log_2\|(a,b)\|_\infty\)
is roughly half of \(\log_2N\).
(We remove the ``If'' and the
``roughly'' for our \(\psi\) in \S\ref{sec:decompositions}.)
The endomorphism lets us 
replace conventional \(\log_2N\)-bit scalar multiplications 
with \(\frac{1}{2}\log_2N\)-bit multiexponentiations.
In terms of basic binary methods,
we are halving the loop length,
cutting the number of doublings in half.

Of course, in practice we are not halving the execution time.
The precise speedup ratio 
depends on a variety of factors,
including
the choice of exponentiation and multiexponentiation
algorithms, the cost of computing \(\psi\), 
the shortness of \(a\) and \(b\) on the average,
and the cost of doublings and addings in terms of bit operations---to 
say nothing of the cryptographic protocol,
which may prohibit some other conventional speedups.
For example: 
in~\cite{GLS}, Galbraith, Lin, and Scott 
report experiments where cryptographic operations
on GLS curves required between 70\% and 83\% of the time required for 
the previous best practice curves---with the variation
depending on the architecture, 
the underyling point arithmetic, and the protocol.

To put this technique into practice,
we need a source of cryptographic elliptic curves 
equipped with efficient endomorphisms.
To date, in the large characteristic case\footnote{
    We are primarily interested in the large characteristic case,
    where \(q = p\) or \(p^2\);
    so we will not discuss \(\tau\)-adic/Frobenius expansion-style
    techniques here.
},
there have been essentially only two constructions:
\begin{enumerate}
    \item The classic \emph{Gallant--Lambert--Vanstone} (GLV)
        construction~\cite{GLV}.
        Elliptic curves over number fields
        with explicit complex multiplication by 
        CM-orders with small discriminants
        are reduced modulo suitable primes~\(p\);
        an explicit endomorphism on the CM curve
        reduces to an efficient endomorphism over the finite field.
    \item The more recent 
        \emph{Galbraith--Lin--Scott} (GLS) construction~\cite{GLS}.
        Here, curves over \(\FF_p\) are viewed over \(\FF_{p^2}\);
        the \(p\)-power sub-Frobenius induces an extremely efficient
        endomorphism on the quadratic twist
        (which can have prime order).
\end{enumerate}

These two constructions have since been 
combined to give 3- and 4-dimensional variants~\cite{Longa--Sica,ZHXS},
and extended to hyperelliptic curves 
in a variety of ways~\cite{BCHL,Kohel--Smith,SCQ,Takashima}.
However, basic GLV and GLS remain the archetypal constructions.

\medskip
\noindent
\underline{\emph{Our contribution: new families of endomorphisms.}}\ \ 
In this work,
we propose a new source of elliptic curves over \(\FF_{p^2}\)
with efficient endomorphisms: quadratic \(\QQ\)-curves.
\begin{definition}
    \label{def:QQ-curves}
    A \emph{quadratic \(\QQ\)-curve of degree \(d\)}
    is an elliptic curve \(\EC\) 
    \emph{without} complex multiplication,
    defined over a quadratic number field
    \(K\),
    such that there exists an isogeny of degree \(d\)
    from \(\EC\) to its Galois conjugate \(\conj{\EC}\),
    where \(\subgrp{\sigma} = \Gal(K/\QQ)\).
    (The Galois conjugate \(\conj{\EC}\)
    is the curve formed by applying \(\sigma\) to all of the
    coefficients of \(\EC\).)
\end{definition}

\(\QQ\)-curves are well-established objects of interest in number
theory, where they have formed a natural setting for 
generalizations of the Shimura--Taniyama conjecture.
Ellenberg's survey~\cite{Ellenberg}
gives an excellent introduction to this beautiful theory.
 
Our application of quadratic \(\QQ\)-curves is rather more prosaic:
given a \(d\)-isogeny \(\ECK \to \conj{\ECK}\) over a quadratic field,
we reduce modulo an inert prime \(p\)
to obtain an isogeny \(\EC\to\conj{\EC}\) over \(\FF_{p^2}\).
We then exploit the fact that the \(p\)-power Frobenius isogeny
maps \(\conj{\EC}\) back onto \(\EC\);
composing with the reduced \(d\)-isogeny,
we obtain an endomorphism of \(\EC\)
of degree \(dp\).
For efficiency reasons,
\(d\) must be small;
it turns out that for small values of \(d\),
we can write down one-parameter families of \(\QQ\)-curves
(our approach below was inspired by the explicit techniques of
Hasegawa~\cite{Hasegawa}).
We thus obtain one-parameter families of elliptic curves 
over~\(\FF_{p^2}\) equipped with efficient non-integer endomorphisms.
For these endomorphisms
we can give convenient explicit formul\ae{} 
for short scalar decompositions (see \S\ref{sec:decompositions}).

For concrete examples,
we concentrate on the cases \(d = 2\) and \(3\) (in
\S\ref{sec:degree-2} and~\S\ref{sec:degree-3}, respectively),
where the endomorphism is more efficient than a single
doubling (we briefly discuss higher degrees in~\S\ref{sec:further}).
For maximum generality and flexibility,
we define our curves in short Weierstrass form;
but we include
transformations to 
Montgomery, 
twisted Edwards,
and
Doche--Icart--Kohel
models where appropriate in~\S\ref{sec:models}.

\smallskip
\noindent
\underline{\emph{Comparison with GLV.}}\ \ 
Like GLV, our method involves reducing curves defined over number fields
to obtain curves over finite fields with explicit complex multiplication.
However, we emphasise a profound difference:
in our method, the curves over the number fields 
generally \emph{do not have complex multiplication themselves}.

GLV curves are necessarily isolated examples---and the really
useful examples are extremely limited in number
(see~\cite[App.~A]{Longa--Sica} for a list of curves).
The scarcity of GLV curves\footnote{
    The scarcity of useful GLV curves is easily explained:
    efficient \emph{separable} endomorphisms have
    extremely small degree (so that the dense defining polynomials
    can be evaluated quickly).
    But the degree of the endomorphism is the norm of the corresponding element
    of the CM-order; and to have non-integers of very small norm, 
    the CM-order must 
    have a tiny discriminant.
    Up to twists, the number of elliptic curves 
    with CM discriminant \(D\) is 
    the Kronecker class number \(h(D)\),
    which is in \(\Oh(\sqrt{D})\).
    Of course, for the tiny values of \(D\) in question,
    the asymptotics of \(h(D)\) are irrelevant;
    for the six \(D\) corresponding to endomorphisms of degree at
    most 3, we have \(h(D) = 1\),
    so there is only one \(j\)-invariant.
    For \(D = -4\) (corresponding to \(j = 1728\))
    there are two or four twists over \(\FF_p\);
    for \(D = -3\) (corresponding to \(j = 0\))
    we have two or six,
    and otherwise we have only two.
    In particular,
    there are at most 18 distinct curves over \(\FF_p\)
    with a non-integer endomorphism of degree at most 3.
}
is their Achilles' heel:
as noted in~\cite{GLS},
if~\(p\) is fixed
then there is no guarantee that there will exist a GLV curve 
with prime (or almost-prime) order over~\(\FF_{p}\).
Consider the situation discussed
in~\cite[\S 1]{GLS}:
the most efficient GLV curves have CM discriminants \(-3\) and \(-4\).
If we are working at a 128-bit security level,
then the choice \(p = 2^{255}-19\)
allows particularly fast arithmetic in \(\FF_p\).
But the largest prime factor of the order of a curve
over \(\FF_p\)
with CM discriminant \(-4\) (resp.~\(-3\))
has \(239\) (resp.~\(230\)) bits:
using these curves wastes 9 (resp.~13) potential bits of security.
In fact, we are lucky with \(D = -3\) and \(-4\):
for all of the other discriminants offering endomorphisms of degree at most 3,
we can do no better than a 95-bit prime factor,
which represents a catastrophic 80-bit loss of relative security.

In contrast, our construction yields true families of curves,
covering \(\sim p\) isomorphism classes over \(\FF_{p^2}\).
This gives us
a vastly higher probability of finding
prime (or almost-prime)-order curves over practically important fields.

\smallskip
\noindent
\underline{\emph{Comparison with GLS.}}\ \ 
Like GLS, we construct curves over \(\FF_{p^2}\)
equipped with an inseparable endomorphism.
While these curves are not defined over the prime field,
the fact that the extension degree is only 2 means that Weil descent
attacks offer no advantage when solving DLP instances
(see~\cite[\S9]{GLS}).
And like GLS, our families offer around \(p\) distinct isomorphism
classes of curves, making it easy to find secure group orders
when \(p\) is fixed.

But unlike GLS, our curves have \(j\)-invariants in \(\FF_{p^2}\):
they are not isomorphic to or twists of subfield curves.
This allows us to find twist-secure curves,
which are resistant to the Fouque--Lercier--R\'eal--Valette 
fault attack~\cite{Fouque--Lercier--Real--Valette}.
As we will see in \S\ref{sec:GLS},
our construction reduces to GLS in the degenerate case \(d = 1\)
(that is, where \(\phiK\) is an isomorphism).
Our construction is therefore a sort of 
generalized GLS---though it is not the higher-degree generalization 
anticipated by Galbraith, Lin, and Scott themselves,
which composes the sub-Frobenius with a non-rational separable
homomorphism and its dual homomorphism
(cf.~\cite[Theorem~1]{GLS}).

In \S\ref{sec:decompositions},
we prove that we can immediately obtain decompositions 
of the same bitlength as GLS for curves over the same fields:
the decompositions produced by
our Proposition~\ref{prop:bitlength}
are identical to the GLS decompositions 
of~\cite[Lemma~2]{GLS} when \(d = 1\),
up to sign.
For this reason, we do not provide extensive implementation details in
this paper:
while our endomorphisms cost a few more \(\FF_{q}\)-operations to
evaluate than the GLS endomorphism, this evaluation is typically carried
out only once per scalar multiplication. 
This evaluation is the only difference between a GLS scalar
multiplication and one of ours:
the subsequent multiexponentiations have exactly the same length 
as in GLS, and the underlying curve and field arithmetic is the same, too.

\section{
    Notation and conventions
}

Throughout, we work over fields of characteristic not 2 or 3.
Let 
\[
    \EC: y^2 = x^3 + a_2x^2 + a_4x + a_6
\]
be an elliptic curve over such a field \(K\).

\smallskip
\noindent
\underline{\emph{Galois conjugates.}}\ \ 
For every automorphism \(\sigma\) of \(K\),
we define the conjugate curve 
\[
    \conj{\EC}: y^2 = x^3 + \conj{a_2}x^2 + \conj{a_4}x + \conj{a_6} .
\]
If \(\phi: \EC\to\EC_1\) is an isogeny,
then 
we obtain a conjugate isogeny
\(\conj{\phi}:\conj{\EC}\to\conj{\EC_1}\)
by
applying \(\sigma\) 
to the defining equations of \(\phi\), \(\EC\), and \(\EC_1\).

\smallskip
\noindent
\underline{\emph{Quadratic twists.}}\ \ 
For every \(\lambda \not=0\) in \(\overline{K}^\times\),
we define a twisting isomorphism
\[
    \twistiso{\lambda}:
    \EC
    \longrightarrow
    \twist{\lambda}{\EC}: 
    y^2 = x^3 + \lambda^2a_2x^2 + \lambda^4a_4x + \lambda^6a_6
\]
by
\[
    \twistiso{\lambda}
    : 
    (x,y) 
    \longmapsto
    (\lambda^2 x, \lambda^3 y) 
    .
\]
The twist \(\twist{\lambda}{\EC}\) is defined over
\(K(\lambda^2)\), 
and \(\twistiso{\lambda}\) is defined over \(K(\lambda)\).
For every \(K\)-endo\-morphism \(\psi\) of \(\EC\),
there is a \(K(\lambda^2)\)-endomorphism
\(
    \twist{\lambda}{\psi} 
    =
    \twistiso{\lambda}\psi\twistiso{\lambda^{-1}}
\) 
of
\(\twist{\lambda}{\EC}\).
Observe that 
\(
    \twistiso{\lambda_1}\twistiso{\lambda_2}
    =
    \twistiso{\lambda_1\lambda_2}
\)
for any \(\lambda_1,\lambda_2\) in \(K\),
and 
\(\twistiso{-1} = [-1]\).
Also,
\(\conj{(\twist{\lambda}{\EC})} =
\twist{\conj{\lambda}}{\conj{\EC}}\)
for all automorphisms \(\sigma\) of \(\Kbar\).

If \(\mu\) is a nonsquare in \(K\),
then \(\twist{\sqrt{\mu}}{\EC}\) is called a \emph{quadratic twist}.
If \(K = \FF_{q}\), then \(\twist{\sqrt{\mu_1}}{\EC}\) and 
\(\twist{\sqrt{\mu_2}}{\EC}\) are \(\FF_{q}\)-isomorphic 
for all nonsquares \(\mu_1\), \(\mu_2\) in \(\FF_{q}\)
(the isomorphism \(\twistiso{\sqrt{\mu_1/\mu_2}}\) is defined over
\(\FF_{q}\) because \(\mu_1/\mu_2\) must be a square).
When the choice of nonsquare is not important,
we denote the quadratic twist by~\(\EC'\).
Similarly,
if \(\psi\) is an \(\FF_{q}\)-endomorphism of \(\EC\),
then \(\psi'\)
is the corresponding \(\FF_{q}\)-endomorphism of \(\EC'\).
(\emph{Conjugates are marked by left-superscripts,
twists by right-superscripts}.)

\smallskip
\noindent
\underline{\emph{The trace.}}\ \ 
If \(K = \FF_{q}\),
then \(\pi_{\EC}\) denotes 
the \(q\)-power Frobenius endomorphism of \(\EC\).
Recall that
the characteristic polynomial
of \(\pi_{\EC}\) has the form
\[
    \chi_{\EC}(T) = T^2 - \trace{\EC}T + q 
    ,
    \qquad \text{with} \qquad 
    |\trace{\EC}| \le 2\sqrt{q}
    .
\]
The integer
\(\trace{\EC}\) is the \emph{trace} of \(\EC\);
we have
\(\#\EC(\FF_{q}) = q + 1 - \trace{\EC}\)
and \(\trace{\EC'} = -\trace{\EC}\).

\smallskip
\noindent
\underline{\emph{\(p\)-th powering.}}\ \ 
We write \((p)\) for 
the \(p\)-th powering automorphism of \(\FFbar_p\). 
Note that \((p)\) is almost trivial to compute
on \(\FF_{p^2} = \FF_{p}(\sqrt{\Delta})\),
because
\( \conj[(p)]{(a + b\sqrt{\Delta})} = a - b\sqrt{\Delta} \)
for all \(a\) and \(b\) in \(\FF_{p}\).

\section{
    Quadratic \(\QQ\)-curves and their reductions
}
\label{sec:construction}

Suppose \(\ECK/\QQ(\sqrt{\Delta})\) is a quadratic \(\QQ\)-curve 
of prime degree \(d\) (cf.~Definition~\ref{def:QQ-curves}),
where \(\Delta\) is a discriminant prime to \(d\),
and let
\(\phiK: \ECK \to \conj{\ECK}\)
be the corresponding \(d\)-isogeny.
In general, \(\phiK\) is only defined over a quadratic
extension \(\QQ(\sqrt{\Delta},\gamma)\)
of \(\QQ(\sqrt{\Delta})\).
We can compute \(\gamma\) from \(\Delta\) and \(\ker\phiK\)
using~\cite[Proposition 3.1]{Gonzalez}, but after a suitable twist
we can always reduce to the case where 
\(\gamma = \sqrt{\pm d}\)
(see~\cite[remark after Lemma~3.2]{Gonzalez}).
The families of explicit \(\QQ\)-curves of degree~\(d\) 
that we treat below
have their isogenies defined over \(\QQ(\sqrt{\Delta},\sqrt{-d})\);
so to simplify matters, from now on we will 
\begin{center}
    Assume \(\phiK\) is defined over \(\QQ(\sqrt{\Delta},\sqrt{-d})\).
\end{center}

Let \(p\) be a prime of good reduction for \(\ECK\)
that is inert in \(\QQ(\sqrt{\Delta})\) and prime to \(d\).
If \(\mathcal{O}_\Delta\) is the ring of integers of \(\QQ(\sqrt{\Delta})\),
then 
\[
    \FF_{p^2} = \mathcal{O}_\Delta/(p) = \FF_p(\sqrt{\Delta}) .
\]
Looking at the Galois groups of our fields, we have a series of
injections
\[
    \subgrp{(p)}
    =
    \Gal(\FF_p(\sqrt{\Delta})/\FF_p) 
    \hookrightarrow 
    \Gal(\QQ(\sqrt{\Delta})/\QQ) 
    \hookrightarrow
    \Gal(\QQ(\sqrt{\Delta},\sqrt{-d})/\QQ)
    .
\]
The image of \((p)\) in \(\Gal(\QQ(\sqrt{\Delta})/\QQ)\) is \(\sigma\),
because \(p\) is inert in \(\QQ(\sqrt{\Delta})\).
When extending \(\sigma\) to 
an automorphism of \(\QQ(\sqrt{\Delta},\sqrt{-d})\),
we extend it to be the image of \((p)\):
that is,
\begin{equation}
    \label{eq:FFp2-conj}
    \conj{\left(
        \alpha + \beta\sqrt{\Delta} + \gamma\sqrt{-d} + \delta\sqrt{-d\Delta}
    \right)} 
    = 
    \alpha - \beta\sqrt{\Delta} 
        + \Legendre{-d}{p}\left(\gamma\sqrt{-d} - \delta\sqrt{-d\Delta}\right)
\end{equation}
for all \(\alpha, \beta, \gamma\), and \(\delta \in \QQ\).
(Recall that the Legendre symbol \(\Legendre{n}{p}\)
is \(1\) if \(n\) is a square mod \(p\), \(-1\) if \(n\) is not a square
mod \(p\), and \(0\) if \(p\) divides \(n\).)

Now let \(\EC\)
be the reduction modulo \(p\) of \(\ECK\).
The curve \(\conj{\ECK}\) reduces to \(\conj[(p)]{\EC}\),
while the \(d\)-isogeny \(\phiK: \ECK\to\conj{\ECK}\)
reduces to a \(d\)-isogeny
\(\phi: \EC \to \conj[(p)]{\EC}\)
defined over \(\FF_{p^2}\).

Applying \(\sigma\) to \(\phiK\),
we obtain a second \(d\)-isogeny
\(\conj{\phiK}:\conj{\ECK}\to\ECK\)
travelling in the opposite direction,
which reduces mod \(p\) to
a conjugate isogeny \(\conj[(p)]{\phi}: \conj[(p)]{\EC} \to \EC\)
over~\(\FF_{p^2}\).
Composing \(\conj{\phiK}\) with \(\phiK\)
yields endomorphisms \(\conj{\phiK}\circ\phiK\) of \(\ECK\)
and \(\phiK\circ\conj{\phiK}\) of \(\conj{\ECK}\),
each of degree~\(d^2\).
But (by definition) \(\ECK\) and \(\conj{\ECK}\)
do not have complex multiplication,
so all of their endomorphisms are integer multiplications;
and since the only integer multiplications of degree \(d^2\)
are \([d]\) and \([-d]\), we can conclude that
\[
    \conj{\phiK}\circ\phiK = [\epsilon_p d]_{\ECK}
    \quad
    \text{ and }
    \quad
    \phiK\circ\conj{\phiK} = [\epsilon_p d]_{\conj{\ECK}},
    \quad 
    \text{ where }
    \epsilon_p \in \{\pm1\}
    .
\]
Technically, \(\conj{\phiK}\) and \(\conj[(p)]{\phi}\)
are---\emph{up to sign}---the dual isogenies of \(\phiK\) and \(\phi\),
respectively.
The sign \(\epsilon_p\) depends on \(p\) (as well as on \(\phiK\)):
if \(\tau\) is the extension of \(\sigma\) to
\(\QQ(\sqrt{\Delta},\sqrt{-d})\) that is \emph{not} the image of \((p)\),
then \(\conj[\tau]{\phiK}\circ\phiK = [-\epsilon_p d]_{\ECK}\).
Reducing mod \(p\),
we see that
\begin{equation}
    \label{eq:pphiphi}
    \conj[(p)]{\phi}\circ\phi 
    =
    [\epsilon_p d]_{\EC}
    \quad 
    \text{ and }
    \quad 
    \phi\circ\conj[(p)]{\phi} 
    =
    [\epsilon_p d]_{\conj[(p)]\EC}
    .
\end{equation}

The map
\((x,y)\mapsto(x^p,y^p)\)
defines \(p\)-isogenies
\[
    \pi_0: \conj[(p)]{\EC} \longrightarrow \EC 
    \quad \text{ and }\quad 
    \conj[(p)]{\pi_0}: \EC \longrightarrow \conj[(p)]{\EC}
    .
\]
Clearly,
\(\conj[(p)]{\pi_0}\circ\pi_0\)
(resp.~\(\pi_0\circ\conj[(p)]{\pi_0}\))
is the \(p^2\)-power Frobenius endomorphism of~\(\EC\)
(resp.~\(\conj[(p)]{\EC}\)).
Composing \(\pi_0\) with \(\phi\) yields a degree-\(pd\) endomorphism
\[
    \psi := \pi_0\circ\phi \in \End(\EC) .
\]
If \(d\) is very small---say, less than \(10\)---then 
\(\psi\) is efficient 
because \(\phi\) is defined by polynomials of degree about \(d\),
and \(\pi_0\) acts as a simple conjugation on coordinates in
\(\FF_{p^2}\), as in Eq.~\eqref{eq:FFp2-conj}.
(The efficiency of \(\psi\) depends primarily on its separable
degree, \(d\), and not on the inseparable part \(p\).)

We also obtain an endomorphism \(\psi'\)
on the quadratic twist \(\EC'\) of \(\EC\).  
Indeed, if \(\EC' = \twist{\sqrt{\mu}}{\EC}\),
then \(\psi' = \twist{\sqrt{\mu}}{\psi}\),
and \(\psi'\) is defined over \(\FF_{p^2}\).

\begin{proposition}
    \label{prop:eigenvalue}
    With the notation above:
    \begin{align*}
        \psi^2 = [\epsilon_p d]\pi_{\EC}
        \qquad
        & \text{ and }
        \qquad
        (\psi')^2 = [-\epsilon_p d]\pi_{\EC'}
        .
        \intertext{
            There exists an integer \(r\) satisfying
            \( dr^2 = 2p + \epsilon_p \trace{\EC} \)
            such that
        }
        \psi
        =
        \tfrac{1}{r}\left(\pi_{\EC} + \epsilon_p p\right)
        \qquad
        & \text{ and }
        \qquad
        \psi'
        =
        \tfrac{-1}{r}\left(\pi_{\EC'} - \epsilon_p p\right)
        .
    \end{align*}
    The characteristic polynomial of both \(\psi\) and \(\psi'\)
    is
    \[
        P_\psi(T)
        =
        P_{\psi'}(T)
        =
        T^2 - \epsilon_p rdT + dp 
        .
    \]
\end{proposition}
\begin{proof}
    Clearly, \( \pi_0\circ\phi = \conj[(p)]{\phi}\circ\conj[(p)]{\pi_0} \).
    Hence 
    \[
        \psi^2 
        = 
        \pi_0\circ\phi\circ\pi_0\circ\phi 
        =
        \pi_0\circ\phi\circ\conj[(p)]{\phi}\circ\conj[(p)]{\pi_0}
        =
        \pi_0[\epsilon_p d]\conj[(p)]{\pi_0}
        =
        [\epsilon_p d]\pi_0\conj[(p)]{\pi_0}
        =
        [\epsilon_p d]\pi_{\EC}
        .
    \]
    Choosing a nonsquare \(\mu\) in \(\FF_{p^2}\), so \(\EC' =
    \twist{\sqrt{\mu}}{\EC}\) and \(\psi' = \twist{\sqrt{\mu}}{\psi}\), 
    we find
    \begin{align*}
        (\psi')^2 
        = 
        \twistiso{\mu^{1/2}}\circ\psi^2\circ\twistiso{\mu^{-1/2}}
        & =
        \twistiso{\mu^{1/2}}\circ[\epsilon_p d]\pi_{\EC}\circ\twistiso{\mu^{-1/2}}
        \\
        & =
        \twistiso{\mu^{(1-p^2)/2}}[\epsilon_p d]\pi_{\EC'}
        =
        \twistiso{-1}[\epsilon_p d]\pi_{\EC'}
        =
        [-\epsilon_p d]\pi_{\EC'}
    \end{align*}
    Using the relations 
    \(\pi_{\EC}^2 - \trace{\EC}\pi_{\EC} + p^2 = 0\)
    and
    \(\pi_{\EC'}^2 + \trace{\EC}\pi_{\EC'} + p^2 = 0\),
    we verify that
    the expressions for \(\psi\) and \(\psi'\)
    give the two square roots
    of \(\epsilon_p d\pi_{\EC}\) in \(\QQ(\pi_{\EC})\),
    and \(-\epsilon_p d\pi_{\EC'}\) in \(\QQ(\pi_{\EC}')\),
    and that the claimed characteristic polynomial is satisfied.
    \qed
\end{proof}

Now we just need a source of quadratic \(\QQ\)-curves of small degree.
Elkies~\cite{Elkies}
shows that all \(\QQ\)-curves
correspond to rational points on certain modular curves:
Let \(X^*(d)\) be the quotient of
the modular curve \(X_0(d)\)
by all of its Atkin--Lehner involutions,
let \(K\) be a quadratic field,
and let \(\sigma\) be the involution of \(K\) over \(\QQ\).
If \(e\) is a point in \(X^*(d)(\QQ)\)
and \(E\) is a preimage of \(e\) in \(X_0(d)(K)\setminus X_0(d)(\QQ)\),
then \(E\) parametrizes (up to \(\QQbar\)-isomorphism) a \(d\)-isogeny
\(\phiK: \ECK \to 
\conj{\ECK}\) over \(K\).

Luckily enough, for very small \(d\), 
the curves \(X_0(d)\) and \(X^*(d)\) have genus zero---so not only do we
get plenty of rational points on \(X^*(d)\), 
we get a whole one-parameter family of
\(\QQ\)-curves of degree~\(d\).
Hasegawa
gives explicit universal curves for \(d = 2\), \(3\), and~\(7\)
in~\cite[Theorem 2.2]{Hasegawa}:
for each squarefree integer \(\Delta \not=1\),
every \(\QQ\)-curve of degree \(d = 2, 3, 7\) over \(\QQ(\sqrt{\Delta})\)
is \(\QQbar\)-isomorphic to a rational specialization of one of these
families.
Hasegawa's curves for \(d = 2\) and \(3\)
(\(\ECK_{2,\Delta,s}\) in~\S\ref{sec:degree-2}
and
\(\ECK_{3,\Delta,s}\) in~\S\ref{sec:degree-3})
suffice not only to illustrate our ideas,
but also to give useful practical examples.

\section{
    Short scalar decompositions
}
\label{sec:decompositions}

Before moving on to concrete constructions,
we will show that the endomorphisms developed in
\S\ref{sec:construction} yield short scalar decompositions.
Proposition~\ref{prop:bitlength} below
gives explicit formul\ae{} for producing decompositions
of at most \(\lceil{log_2p\rceil}\) bits.

Suppose \(\G\) is a cyclic subgroup of \(\EC(\FF_{p^2})\) 
such that \(\psi(\G) = \G\),
and let \(N = \#\G\).
Proposition~\ref{prop:eigenvalue}
shows that
\(\psi\) acts as a square root of \(\epsilon_p d\) on \(\G\):
its eigenvalue is
\begin{equation}
    \label{eq:eigenvalue}
    \lambda_\psi \equiv (1 + \epsilon_p p)/r \pmod{N} .
\end{equation}
We want to 
compute a decomposition
\[
    m = a + b\lambda_\psi \pmod N
\]
so as to efficiently compute 
\[
    [m]P = [a]P + [b\lambda_\psi]P = [a]P + [b]\psi(P) .
\]

The decomposition of \(m\) 
is not unique: far from it.
The set of all decompositions \((a,b)\) of \(m\)
is the coset \((m,0) + \Lattice\),
where
\[
    \Lattice
    := 
    \subgrp{
        (N,0),(-\lambda_\psi,1)
    }
    \subset \ZZ^2
\]
is the lattice of decompositions of \(0\)
(that is, of \((a,b)\) such that \(a + b\lambda_\psi \equiv 0\pmod{N}\)).

We want to find a decomposition where \(a\) and \(b\)
have minimal bitlength:
that is, where \(\lceil\log_2\|(a,b)\|_\infty\rceil\) is as small as
possible.
The standard technique is to (pre)-compute a short basis of \(\Lattice\),
then use Babai rounding~\cite{Babai}
to transform each scalar~\(m\)
into a short decomposition \((a,b)\).
The following lemma outlines this process;
for further detail and analysis,
see 
\cite[\S4]{GLV} and~\cite[\S18.2]{Galbraith}.

\begin{lemma}
    \label{lemma:decomp}
    Let 
    \(\vec{e}_1,\vec{e}_2\)
    be linearly independent vectors in \(\Lattice\).
    Let \(m\) be an integer,
    and set
    \[
        (a,b) 
        := 
        (m,0) - \roundoff{\alpha}\vec{e}_1 - \roundoff{\beta}\vec{e}_2 
        ,
    \]
    where
    \((\alpha,\beta)\) is the (unique) solution in \(\QQ^2\)
    to the linear system
    \((m,0) = \alpha\vec{e}_1 + \beta\vec{e}_2\).
    Then 
    \[
        m \equiv a + \lambda_\psi b \pmod{N}
        \qquad\text{and}\qquad
        \|(a,b)\|_\infty
        \le
        \max\left(\|\vec{e}_1\|_\infty,\|\vec{e}_2\|_\infty\right)
        .
    \]
\end{lemma}
\begin{proof}
    This is just~\cite[Lemma~2]{GLV} (under the infinity norm). 
    \qed
\end{proof}
 
We see that better decompositions of \(m\)
correspond to shorter bases for~\(\Lattice\).
If \(|\lambda_\psi|\) is not unusually small,
then we can compute a basis for \(\Lattice\)
of size \(\Oh(\sqrt{N})\)
using the Gauss reduction or Euclidean algorithms
(cf.~\cite[\S4]{GLV} and~\cite[\S17.1.1]{Galbraith}).\footnote{
    Bounds on the constant hidden by the \(\Oh(\sqrt{N})\)
    are derived in~\cite{SCQ},
    but they are suboptimal for our endomorphisms.
    In cryptographic contexts, Proposition~\ref{prop:bitlength}
    gives better results.
}
The basis depends only on \(N\) and \(\lambda_\psi\),
so it can be precomputed.

In our case,
lattice reduction is unnecessary:
we can immediately write down two linearly independent vectors in \(\Lattice\)
that are ``short enough'',
and thus give explicit formulae for \((a,b)\) in terms of \(m\).
These decompositions have length \(\lceil{\log_2p}\rceil\),
which is near-optimal in cryptographic contexts:
if \(N\sim\#\EC(\FF_{p^2}) \sim p^2\),
then \(\log_2 p \sim \frac{1}{2}\log_2N\).

\begin{proposition}
    \label{prop:bitlength}
    With the notation above:
    given an integer \(m\),
    let
    \begin{align*}
        a 
        & = 
        m - \bigroundoff{m(1 + \epsilon_p p)/\#\EC(\FF_{p^2})}(1 + \epsilon_p p)
            + \bigroundoff{mr/\#\EC(\FF_{p^2})}\epsilon_pdr
        \qquad\text{and}\\
        b
        & =
        \bigroundoff{m(1 + \epsilon_p p)/\#\EC(\FF_{p^2})}r
            - \bigroundoff{mr/\#\EC(\FF_{p^2})}(1 + \epsilon_p p)
        .
    \end{align*}
    Then, assuming \(d \ll p\) and \(m \not\equiv0\pmod{N}\), we have
    \[
        m \equiv a + b\lambda_\psi \pmod{N}
        \qquad
        \text{ and }
        \qquad
        \lceil\log_2\|(a,b)\|_\infty\rceil
        \le 
        \lceil\log_2p\rceil
        .
    \]
\end{proposition}
\begin{proof}
    Eq.~\eqref{eq:eigenvalue} 
    yields
    \(r\lambda_\psi \equiv 1 + \epsilon_p p \pmod{N}\)
    and \(r\epsilon_pd \equiv (1 + \epsilon_p p)\lambda_\psi \pmod{N}\),
    so 
    the vectors
    \(\vec{e}_1 = (1 + \epsilon_p p, -r)\) 
    and \(\vec{e}_2 = (-\epsilon_p dr, 1 + \epsilon_p p)\)
    are in \(\Lattice\)
    (they generate a sublattice of determinant \(\#\EC(\FF_{p^2})\)).
    Applying Lemma~\ref{lemma:decomp}
    with \(\alpha = m(1+\epsilon_p p)/\#\EC(\FF_{p^2})\)
    and \(\beta = mr/\#\EC(\FF_{p^2})\),
    we see that 
    \(m \equiv a + b\lambda_\psi\pmod{N}\)
    and \(\|(a,b)\|_\infty \le \|\vec{e}_2\|_\infty\).
    But \(d|r|\le 2\sqrt{dp}\)
    (since \(|\trace{\EC}| \le 2p\))
    and \(d \ll p\), 
    so \(\|\vec{e}_2\|_\infty = p + \epsilon_p\).
    The result follows on taking logs,
    and noting that 
    \(\lceil{\log_2(p\pm1)}\rceil = \lceil{\log_2 p}\rceil\)
    (since \(p > 3\)).
    \qed
\end{proof}

\section{
    Endomorphisms from quadratic \(\QQ\)-curves of degree 2
}
\label{sec:degree-2}

Let \(\Delta\) be a squarefree integer.
Hasegawa
defines a one-parameter family of elliptic curves
over \(\QQ(\sqrt{\Delta})\) by
\begin{equation}\label{eq:Hasegawa-2}
    \ECK_{2,\Delta,s}: 
    y^2 
    = 
    x^3 - 6(5 - 3s\sqrt{\Delta})x + 8(7 - 9s\sqrt{\Delta})
    ,
\end{equation}
where \(s\) is a free parameter taking values in \(\QQ\)
\cite[Theorem 2.2]{Hasegawa}.
The discriminant of \(\ECK_{2,\Delta,s}\) is
\(2^9\cdot3^6(1 - s^2\Delta)(1 + s\sqrt{\Delta})\),
so the curve \(\ECK_{2,\Delta,s}\) 
has good reduction at every \(p > 3\) with \(\Legendre{\Delta}{p} = -1\),
for every \(s\) in \(\QQ\).

The curve 
\(\ECK_{2,\Delta,s}\)
has a rational \(2\)-torsion point \((4,0)\),
which generates the kernel of a \(2\)-isogeny 
\(\phiK_{2,\Delta,s}: \ECK_{2,\Delta,s} \to \conj{\ECK_{2,\Delta,s}}\)
defined over \(\QQ(\sqrt{\Delta},\sqrt{-2})\).
We construct \(\phiK_{2,\Delta,s}\) explicitly:
V\'elu's formulae~\cite{Velu} 
define
the (normalized) quotient 
\(\ECK_{2,\Delta,s} \to \ECK_{2,\Delta,s}/\subgrp{(4,0)}\),
and then
the isomorphism \(\ECK_{2,\Delta,s}/\subgrp{(4,0)} \to \conj{\ECK_{2,\Delta,s}}\)
is the quadratic twist \(\twistiso{1/\sqrt{-2}}\).
Composing, we obtain an expression for the isogeny as a rational map:
\[
    \label{eq:phiK2}
    \phiK_{2,\Delta,t}:
    (x,y)
    \longmapsto
    \left(
        \frac{-x}{2} - \frac{9(1 + s\sqrt{\Delta})}{x-4}
        , 
        \frac{y}{\sqrt{-2}}
        \left(
            \frac{-1}{2} + \frac{9(1 + s\sqrt{\Delta})}{(x-4)^2}
        \right)
    \right)
    .
\]
Conjugating and composing,
we recognise that
\( 
    \conj[\sigma]{\phiK_{2,\Delta,t}}\circ\phiK_{2,\Delta,t} 
    = 
    [2]
\) if \(\sigma(\sqrt{-2}) = -\sqrt{-2}\),
and \([-2]\) if \(\sigma(\sqrt{-2}) = \sqrt{-2}\):
that is, the sign function for \(\phiK_{2,\Delta,t}\) is
\begin{equation}
    \label{eq:eps-2}
    \epsilon_p 
    = 
    -\Legendre{-2}{p} 
    =
    \begin{cases}
        1  & \text{if }\ p \equiv 5, 7 \pmod{8} , \\
        -1 & \text{if }\ p \equiv 1, 3 \pmod{8} . \\
    \end{cases}
\end{equation}


\begin{theorem}
    \label{th:d2}
    Let \(p > 3\) be a prime,
    and define \(\epsilon_p\) as in Eq.~\eqref{eq:eps-2}.
    Let \(\Delta\) be a nonsquare\footnote{
        The choice of \(\Delta\) is (theoretically) irrelevant,
        since all quadratic extensions of
        \(\FF_{p}\) are isomorphic.
        If \(\Delta\) and \(\Delta'\) are two nonsquares in \(\FF_{p}\),
        then \(\Delta/\Delta' = a^2\) for some \(a\) in \(\FF_{p}\),
        so \(\EC_{2,\Delta,t}\) and \(\EC_{2,\Delta',at}\)
        are identical.
        We are therefore free to choose any practically convenient
        value for~\(\Delta\),
        such as one permitting faster arithmetic
        in~\(\FF_{p}(\sqrt{\Delta})\).
    }
    in \(\FF_p\),
    so \(\FF_{p^2} = \FF_p(\sqrt{\Delta})\).
    Let \(C_{2,\Delta}: \FF_p \to \FF_{p^2}\) be the mapping defined by
    \[
        C_{2,\Delta}(s) := 9(1 + s\sqrt{\Delta}) .
    \]
    For each \(s\) in \(\FF_p\),
    let \(\EC_{2,\Delta,s}\) 
    be the elliptic curve over \(\FF_{p^2}\) defined by
    \[
        \EC_{2,\Delta,s}
        : 
        y^2 = x^3 + 2(C_{2,\Delta}(s) - 24)x - 8(C_{2,\Delta}(s) - 16)
        .
    \]
    Then \(\EC_{2,\Delta,s}\) 
    has an efficient \(\FF_{p^2}\)-endomorphism
    \[
        \psi_{2,\Delta,s}: 
        (x,y) 
        \longmapsto
        \left(
            \frac{-x^p}{2} - \frac{C_{2,\Delta}(s)^p}{x^p-4}
            ,
            \frac{y^p}{\sqrt{-2}}
            \left(\frac{-1}{2} + \frac{C_{2,\Delta}(s)^p}{(x^p-4)^2}\right)
        \right)
        ,
    \]
    of degree \(2p\),
    such that
    \[
        \psi_{2,\Delta,s}
        =
        \frac{1}{r}\left(\pi_{\EC_{2,\Delta,s}} + \epsilon_p p\right)
        \qquad \text{ and } \qquad 
        \psi_{2,\Delta,s}^2 = [\epsilon_p 2]\pi_{\EC_{2,\Delta,s}} 
    \]
    for some integer \(r\) satisfying
    \( 2r^2 = 2p + \epsilon_p \trace{\EC_{2,\Delta,s}} \).
    The characteristic polynomial of \(\psi_{2,\Delta,s}\)
    is
    \(
        P_{2,\Delta,s}(T)
        =
        T^2 - \epsilon_prT + 2p 
    \).
    The twisted endomorphism \(\psi_{2,\Delta,s}'\)
    on \(\EC_{2,\Delta,s}'\)
    satisfies
    \(
        \psi_{2,\Delta,s}' 
        = 
        (-\pi_{\EC_{2,\Delta,s}'} + \epsilon_p p)/r
    \),
    and
    \((\psi_{2,\Delta,s}')^2 = [-\epsilon_p 2]\pi_{\EC_{2,\Delta,s}'}\),
    and \(P_{2,\Delta,s}(\psi_{2,\Delta,s}') = 0\).
\end{theorem} 
\begin{proof}
    Reduce \(\ECK_{2,\Delta,s}\) and \(\phiK_{2,\Delta,s}\) mod \(p\)
    and compose with \(\pi_0\) as in~\S\ref{sec:construction},
    then apply Proposition~\ref{prop:eigenvalue} 
    using Eq.~\eqref{eq:eps-2}.
    \qed
\end{proof}

If \(\G \subset \EC_{2,\Delta,s}(\FF_{p^2})\) is a cyclic subgroup
of order \(N\)
such that \(\psi_{2,\Delta,s}(\G) = \G\),
then the eigenvalue of \(\psi_{2,\Delta,s}\) on \(\G\) is 
\[
    \lambda_{2,\Delta,s} =
    \frac{1}{r}\left( 1 + \epsilon_p p \right)
    \equiv 
    \pm\sqrt{\epsilon_p 2} 
    \pmod{N}
    .
\]
Applying Proposition~\ref{prop:bitlength},
we can decompose scalar multiplications in \(\G\)
as \([m]P = [a]P + [b]\psi_{2,\Delta,s}(P)\)
where \(a\) and \(b\) have at most \(\lceil\log_2p\rceil\) bits.

\begin{proposition}
    \label{prop:j-d2}
    Theorem~\ref{th:d2} 
    yields at least \(p-3\) non-isomorphic curves
    (and at least \(2p-6\) non-\(\FF_{p^2}\)-isomorphic curves,
    if we count the quadratic twists)
    equipped with efficient endomorphisms.
\end{proposition}
\begin{proof}
    It suffices to show that the \(j\)-invariant
    \(
        j\big(\EC_{2,\Delta,s}\big)
        =
        \frac{
            2^6(5 - 3s\sqrt{\Delta})^3
        }{
            (1 - s^2\Delta)(1 + s\sqrt{\Delta})
        }
    \)
    takes at least \(p-3\) distinct values in \(\FF_{p^2}\)
    as \(s\) ranges over \(\FF_p\).
    If \(j(\EC_{2,\Delta,s}) = j(\EC_{2,\Delta,s_1})\) with \(s_1
    \not=s_2\),
    then \(s_1\) and \(s_2\) satisfy 
    \(F_0(s_1,s_2) - 2\sqrt{\Delta}F_1(s_1,s_2) = 0\),
    where
    \(F_1(s_1,s_2) = (s_1 + s_2)(63\Delta s_1s_2 - 65)\)
    and
    \(F_0(s_1,s_2) = (\Delta s_1s_2 + 1)(81 \Delta s_1s_2 - 175) + 49\Delta(s_1 + s_2)^2\)
    are polynomials over \(\FF_p\).
    If \(s_1\) and \(s_2\) are in \(\FF_p\),
    then we must have \(F_0(s_1,s_2) = F_1(s_1,s_2) = 0\).
    Solving the simultaneous equations, discarding the solutions that can never
    be in \(\FF_p\), and dividing by two (since \((s_1,s_2)\) and
    \((s_2,s_1)\) represent the same collision) yields at most 3
    collisions \(j(\EC_{2,\Delta,s_1}) = j(\EC_{2,\Delta,s_2})\)
    with \(s_1\not=s_2\) in \(\FF_p\).
    \qed
%
\end{proof}

We observe that \(\conj{\ECK_{2,\Delta,s}} = \ECK_{2,\Delta,-s}\),
so we do not gain any more isomorphism classes in
Proposition~\ref{prop:j-d2}
by including the codomain curves.

\section{
    Endomorphisms from quadratic \(\QQ\)-curves of degree 3
}
\label{sec:degree-3}

Let \(\Delta\) be a squarefree discriminant;
Hasegawa defines a one-parameter family of elliptic curves
over \(\QQ(\sqrt{\Delta})\) by
\begin{equation}\label{eq:Hasegawa-3}
    \ECK_{3,\Delta,s}: 
    y^2 =
    x^3 - 3\big(5 + 4s\sqrt{\Delta}\big)x 
    + 2\big(2s^2\Delta + 14s\sqrt{\Delta} + 11\big)
    ,
\end{equation}
where \(s\) is a free parameter taking values in \(\QQ\).
As for the curves in \S\ref{sec:degree-2},
the curve \(\ECK_{3,\Delta,s}\) 
has good reduction at every \(p > 3\) with \(\Legendre{\Delta}{p} = -1\),
for every \(s\) in \(\QQ\).

The curve \(\ECK_{3,\Delta,s}\)
has a subgroup of order~\(3\)
defined by the polynomial \(x - 3\),
consisting of \(0\) and \((3,\pm2(1-s\sqrt{\Delta}))\).
Exactly as in~\S\ref{sec:degree-2},
taking the V\'elu quotient
and twisting by \(1/\sqrt{-3}\)
yields an explicit 3-isogeny 
\(\phiK_{3,\Delta,s}: \ECK_{3,\Delta,s} \to \conj{\ECK_{3,\Delta,s}}\);
its
sign function is
\begin{equation}
    \label{eq:eps-3}
    \epsilon_p = -\Legendre{-3}{p} 
    =
    \begin{cases}
         1 & \text{if }\ p \equiv 2 \pmod{3} , \\
        -1 & \text{if }\ p \equiv 1 \pmod{3} .
    \end{cases}
\end{equation}

\begin{theorem}
    \label{th:d3}
    Let \(p > 3\) be a prime,
    and define \(\epsilon_p\) as in Eq.~\eqref{eq:eps-3}.
    Let \(\Delta\) be a nonsquare\footnote{
        As in Theorem~\ref{th:d2},
        the particular value of \(\Delta\) is theoretically
        irrelevant.
    }
    in \(\FF_{p}\),
    so \(\FF_{p^2} = \FF_p(\sqrt{\Delta})\).
    Let \(C_{3,\Delta}: \FF_p \to \FF_{p^2}\)
    be the mapping defined by
    \[
        C_{3,\Delta}(s) := 2(1 + s\sqrt{\Delta}) .
    \]
    For each \(s\) in \(\FF_p\),
    we let \(\EC_{3,\Delta,s}\) 
    be the elliptic curve over \(\FF_{p^2}\) 
    defined by
    \[
        \EC_{3,\Delta,s}
        : 
        y^2 = x^3 - 3\big(2C_{3,\Delta}(s) + 1\big)x 
            + \big(C_{3,\Delta}(s)^2 + 10C_{3,\Delta}(s) - 2\big)
        .
    \]
    Then
    \(\EC_{3,\Delta,s}\) 
    has an efficient \(\FF_{p^2}\)-endomorphism
    \[
        \psi_{3,\Delta,s}: 
        (x,y) 
        \longmapsto
        \left(
            -\frac{x^p}{3}
            - \frac{4C_{3,\Delta}(s)^p}{x^p-3}
            - \frac{4C_{3,\Delta}(s)^{2p}}{3(x^p-3)^2}
            ,
            \frac{y^p}{\sqrt{-3}}
            \left(
                \frac{-1}{3} 
                + 
                \frac{4C_{3,\Delta}(s)^p}{(x^p-3)^2}
                +
                \frac{
                    8C_{3,\Delta}(s)^{2p}
                }{
                    3(x^p-3)^3
                }
            \right)
        \right)
    \]
    of degree \(3p\),
    such that
    \[
        \psi_{3,\Delta,s}^2 
        = 
        [\epsilon_p3]\pi_{\EC_{3,\Delta,s}}
        \qquad \text{ and } \qquad
        \psi_{3,\Delta,s} 
        = 
        \frac{1}{r}\left(\pi + \epsilon_p p\right)
    \]
    for some integer \(r\) satisfying 
    \(3r^2 = 2p + \epsilon_p\trace{\EC_{3,\Delta,s}}\).
    The characteristic polynomial of 
    \(\psi_{3,\Delta,s}\) 
    is \( P_{3,\Delta,s}(T) = T^2 - \epsilon_prT + 3p \).
    The twisted endomorphism \(\psi_{3,\Delta,s}'\)
    on \(\EC_{3,\Delta,s}'\)
    satisfies
    \((\psi_{3,\Delta,s}')^2 = [-\epsilon_p 3]\pi_{\EC_{3,\Delta,s}'}\),
    and 
    \(
        \psi_{3,\Delta,s}' 
        = 
        (-\pi_{\EC_{3,\Delta,s}'} + \epsilon_p p)/r
    \),
    and 
    \(P_{3,\Delta,s}(\psi_{3,\Delta,s}') = 0\).
\end{theorem} 
\begin{proof}
    Reduce
    \(\ECK_{3,\Delta,s}\)
    and
    \(\phiK_{3,\Delta,s}\)
    mod \(p\),
    compose with \(\pi_0\)
    as in~\S\ref{sec:construction},
    and apply Proposition~\ref{prop:eigenvalue} 
    using Eq.~\eqref{eq:eps-3}.
    \qed
\end{proof}

\begin{proposition}
    \label{prop:j-d3}
    Theorem~\ref{th:d3} 
    yields at least \(p-8\) non-isomorphic curves
    (and counting quadratic twists, at least \(2p-16\)
    non-\(\FF_{p^2}\)-isomorphic curves)
    equipped with efficient endomorphisms.
\end{proposition}
\begin{proof}
    The proof is exactly as for Proposition~\ref{prop:j-d2}.
    \qed
\end{proof}

\section{
    Cryptographic-sized curves
}
\label{sec:crypto-params}

We will now exhibit some curves with our families with cryptographic
parameter sizes, and secure and twist-secure group orders.
We computed the curve orders below
using Magma's implementation of the Schoof--Elkies--Atkin
algorithm~\cite{Schoof,Magma,Magma-Handbook}.

First,
consider the degree-2 curves of \S\ref{sec:degree-2}.
By definition,
\(\EC_{2,\Delta,s}\) and its quadratic twist \(\EC_{2,\Delta,s}'\) 
have points of order 2 over \(\FF_{p^2}\):
they generate the kernels of our endomorphisms.
If \(p \equiv 2\pmod{3}\),
then 
\(2r^2 = 2p + \epsilon_p \trace{\EC}\)
implies \(\trace{\EC} \not\equiv 0 \pmod{3}\),
so 
when \(p \equiv 2\pmod{3}\) either
\(p^2 - \trace{\EC} + 1 = \#\EC_{2,\Delta,s}(\FF_{p^2})\)
or \(p^2 + \trace{\EC} + 1 = \#\EC_{2,\Delta,s}'(\FF_{p^2})\)
is divisible by 3.
However,
when \(p \equiv 1\pmod{3}\)
we can hope to find curves of order twice a prime
whose twist also has order twice a prime.

 
\begin{example}
    Let \(p = 2^{80}-93\) and \(\Delta = 2\).
    For
    \(s = 4556\), we find a twist-secure curve: 
    \(
        \#\EC_{2,2,4556}(\FF_{p^2}) = 2N
    \)
    and 
    \(
        \#\EC_{2,2,4556}'(\FF_{p^2}) = 2N'
    \)
    where
     \begin{align*}
         N  &= 730750818665451459101729015265709251634505119843 \quad
         \text{and}\\
         N' &= 730750818665451459101730957248125446994932083047
     \end{align*}
    are 159-bit primes.
    Proposition~\ref{prop:bitlength}
    lets us replace 160-bit scalar multiplications in 
    \(\EC_{2,2,4556}(\FF_{p^2})\)
    and 
    \(\EC_{2,2,4556}'(\FF_{p^2})\)
    with 80-bit multiexponentiations.
\end{example}


Now, consider the degree-3 curves of \S\ref{sec:degree-3}.
The order of \(\EC_{3,\Delta,s}(\FF_{p^2})\) is always divisible by \(3\):
the kernel of \(\psi_{3,\Delta,s}\)
is generated by the rational point
\((3,C_{3,\Delta}(s))\).
However, on the quadratic twist, 
the nontrivial points in the kernel of \(\psi_{3,\Delta,s}'\) 
are \emph{not} defined over \(\FF_{p^2}\) (they are conjugates over
\(\FF_{p^2}\)),
so \(\EC_{3,\Delta,s}'(\FF_{p^2})\) can have prime order.

\begin{example}
    Let \(p = 2^{127}-1\);
    then \(\Delta = -1\) is a nonsquare in \(\FF_p\).
    The parameter value \(s =
    122912611041315220011572494331480107107\)
    yields 
    \[
        \#\EC_{3,-1,s}(\FF_{p^2}) = 3\cdot N
        \qquad
        \text{ and }
        \qquad
        \#\EC_{3,-1,s}'(\FF_{p^2}) = N' 
        ,
    \]
    where \(N\) is a 253-bit prime
    and \(N'\) is a 254-bit prime.
    Using Proposition~\ref{prop:bitlength},
    any scalar multiplication 
    in \(\EC_{3,-1,s}(\FF_{p^2})\) 
    or \(\EC_{3,-1,s}'(\FF_{p^2})\) 
    can be computed via a 127-bit multiexponentiation.
\end{example}

\begin{example}
    Let \(p = 2^{255}-19\); then \(\Delta = -2\) is a nonsquare in
    \(\FF_p\).
    The parameter
    \(s =
    \mathtt{0x7516D419C4937E5E8F0761FDB9BB0382FE20E9D0B7AB6924BA1DA02561C5145E}
    \)
    \\
    yields \(\#\EC_{3,-2,s}(\FF_{p^2}) = 3\cdot N\)
    and \(\#\EC_{3,-2,s}(\FF_{p^2}) = N'\),
    where \(N\) and \(N'\) are 509- and 510-bit primes, respectively.
    Proposition~\ref{prop:bitlength}
    transforms any 510-bit scalar multiplication
    in \(\EC_{3,-2,s}(\FF_{p^2})\)
    or \(\EC_{2,-2,s}'(\FF_{p^2})\) 
    into a 255-bit multiexponentiation.
\end{example}

\section{
    Montgomery, Twisted Edwards, and Doche--Icart--Kohel models 
}
\label{sec:models}

\smallskip
\noindent
\underline{\emph{Montgomery models.}}\ \ 
The curve \(\EC_{2,\Delta,s}\)
has a Montgomery model over \(\FF_{p^2}\)
if and only if \(2C_{2,\Delta}(s)\) 
is a square in \(\FF_{p^2}\)
(by~\cite[Proposition~1]{OKS}):
in that case, 
setting
\[
    B_{2,\Delta}(s) := \sqrt{2C_{2,\Delta}(s)} 
    \qquad \text{ and } \qquad 
    A_{2,\Delta}(s) = 12/B_{2,\Delta}(s) 
    ,
\]
the birational mapping
\( 
    (x,y) 
    \mapsto 
    (X/Z,Y/Z) 
    =
    \left((x-4)/B_{2,\Delta}(s),y/B_{2,\Delta}(s)^2\right) 
\)
takes us from \(\EC_{2,\Delta,s}\) to the projective Montgomery model
\begin{equation}
    \label{eq:Montgomery}
    \EC_{2,\Delta,s}^\mathrm{M}
    : 
    B_{2,\Delta}(s)Y^2Z = X\left(X^2 + A_{2,\Delta}(s)XZ + Z^2\right) 
    .
\end{equation}
(If \(2C_{2,\Delta}(s)\) is not a square,
then 
\(\EC_{2,\Delta,s}^\mathrm{M}\)
is \(\FF_{p^2}\)-isomorphic to the quadratic twist
\(\EC_{2,\Delta,s}'\).)
These models offer a particularly efficient arithmetic,
where we use only the \(X\) and \(Z\) coordinates~\cite{Montgomery}.
The endomorphism is defined (on the \(X\) and \(Z\) coordinates) 
by
\[
    \psi_{2,\Delta,s}
    :
    (X:Z)
    \longmapsto
    (
        X^{2p} + A_{2,\Delta}(s)^pX^pZ^p + Z^{2p} 
        :
        -2B_{2,\Delta}(s)^{1-p}X^pZ^p
    )
    .
\]


\smallskip
\noindent
\underline{\emph{Twisted Edwards models.}}
\ \ Every Montgomery model corresponds to a twisted Edwards model
(and vice versa) \cite{BBJLP,Hisil--Wong--Carter--Dawson}.
With \(u = X/Z\) and \(v = Y/Z\),
the birational maps
\[
    (u,v) 
    \longmapsto 
    (x_1,x_2) 
    = 
    \left( \frac{u}{v},\frac{u-1}{u+1} \right)
    \qquad
    \text{ and }
    \qquad
    (x_1,x_2)
    \longmapsto
    (u,v) 
    =
    \left(
        \frac{1+x_2}{1-x_2},\frac{1+x_2}{x_1(1-x_2)}
    \right)
\]
take us between the Montgomery model of
Eq.~\eqref{eq:Montgomery}
and the twisted Edwards model
\[
    \EC_{2,\Delta,s}^{\textrm{TE}}:
    a_2(s)x_1^2 + x_2^2 = 1 + d_2(s)x_1^2x_2^2 ,
    \quad
    \text{ where }
    \quad
    \begin{cases}
        a_2(s) & = (A_{2,\Delta}(s) + 2)/B_{2,\Delta}(s) \\
        d_2(s) & = (A_{2,\Delta}(s) - 2)/B_{2,\Delta}(s) .
    \end{cases}
\]

\noindent
\underline{\emph{Doche--Icart--Kohel models.}}\ \ 
Doubling-oriented Doche--Icart--Kohel models of elliptic curves
are defined by equations of the form
\[
    \label{eq:2-DIK}
    y^2 = x(x^2 + D x + 16D) .
\]
These curves have a rational \(2\)-isogeny \(\phi\),
with kernel \(\subgrp{(0,0)}\);
in this form, we can double more quickly
by using the decomposition \([2] = \dualof{\phi}{\phi}\) 
(see~\cite[\S3.1]{DIK} for details).

Our curves \(\EC_{2,\Delta,s}\) come equipped with a rational \(2\)-isogeny,
so it is natural to try putting them in Doche--Icart--Kohel form.
The isomorphism
\[
    \alpha:
    (x,y) 
    \longmapsto 
    (u,v) 
    = 
    \left(
        \mu^2(x+4),
        \mu^3y
    \right)
    \quad
    \text{ with }
    \quad
    \mu = 4\sqrt{6/C_{2,\Delta}(s)}
\]
takes us from
\(\EC_{2,\Delta,s}\) 
into
a doubling-oriented Doche--Icart--Kohel model
\[
    \EC_{2,\Delta,s}^\mathrm{DIK} 
    : 
    v^2 = u\left(u^2 + D_{2,\Delta}(s)u + 16D_{2,\Delta}(s)\right)
    ,
    \quad
    \text{ where }
    \quad 
    D_{2,\Delta}(s) 
    =
    2^7/(1 + s\sqrt{\Delta})
    .
\]
While \(\EC_{2,\Delta,s}^\mathrm{DIK}\)
is defined over \(\FF_{p^2}\),
the isomorphism 
is only defined over~\(\FF_{p^2}(\sqrt{1 + s\sqrt{\Delta}})\);
so if \(1 + s\sqrt{\Delta}\) is not a square in \(\FF_{p^2}\)
then \(\EC_{2,\Delta,s}^\mathrm{DIK}\) 
is 
\(\FF_{p^2}\)-isomorphic to \(\EC_{2,\Delta,s}'\).

The endomorphism 
\(
    \psi_{2,\Delta,s}^\mathrm{DIK} := \alpha\psi_{2,\Delta,s}\alpha^{-1}
\) 
is \(\FFbar_{p}\)-isomorphic to the Doche--Icart--Kohel isogeny,
since they have the same kernel.
The eigenvalue of \(\psi_{2,\Delta,s}\) on cryptographic subgroups
is \(\pm\sqrt{\pm2}\),
so computing \([m]P\) as \([a]P + [b]\psi_{2,\Delta,s}^\mathrm{DIK}\)
with Doche--Icart--Kohel doubling for \([a]\) and \([b]\)
is like using a \(\sqrt{\pm2}\)-adic expansion of \(m\).

Similarly,
we can exploit the rational \(3\)-isogeny on \(\EC_{3,\Delta,s}\)
for Doche--Icart--Kohel tripling (see~\cite[\S3.2]{DIK}).
The isomorphism
\(
    (x,y)
    \mapsto
    (u,v)
    =
    \left(
        a_{3,\Delta}(s)(x/3-1) , b_{3,\Delta}(s)^3y
    \right)
\),
with
\(a_{3,\Delta}(s) = 9/C_{3,\Delta}(s)\)
and \(b_{3,\Delta}(s) = a_{3,\Delta}(s)^{-1/2}\), 
takes us from \(\EC_{3,\Delta,s}\)
to the tripling-oriented Doche--Icart-Kohel model
\[
    \EC_{3,\Delta,s}^\mathrm{DIK} 
    : 
    v^2 = u^3 + 3a_{3,\Delta}(s)(u + 1)^2 
    .
\]

\section{
    Degree one: GLS as a degenerate case
}
\label{sec:GLS}

Returning to the framework of \S\ref{sec:construction},
suppose \(\ECK\) is a curve defined over \(\QQ\),
and base-extended to \(\QQ(\sqrt{D})\):
then \(\ECK = \conj{\ECK}\),
and we can apply the construction of \S\ref{sec:construction}
taking
\(\phiK: \ECK \to \conj{\ECK}\) to be the identity map.
Reducing modulo an inert prime \(p\),
the endomorphism \(\psi\) is nothing but \(\pi_0\)
(which is an endomorphism, since \(\EC\) is a subfield curve).
We have \(\psi^2 = \pi_0^2 = \pi_{\EC}\),
so the eigenvalue of \(\psi\) is \(\pm1\) 
on cryptographic subgroups of
\(\EC(\FF_{p^2})\). 
Clearly, this endomorphism is of no use to us
for scalar decompositions.

However, looking at the quadratic twist \(\EC'\), 
the twisted endomorphism \(\psi'\)
satisfies \((\psi')^2 = -\pi_{\EC'}\);
the eigenvalue of \(\psi'\) on cryptographic subgroups
is a square root of \(-1\).
We have recovered the Galbraith--Lin--Scott endomorphism
(cf.~\cite[Theorem~2]{GLS}).

More generally,
suppose \(\phiK: \ECK \to \conj{\ECK}\)
is a \(\QQbar\)-isomorphism:
that is, an isogeny of degree \(1\).
If \(\ECK\) does not have CM,
then \(\conj{\phiK} = \epsilon_p\phiK^{-1}\),
so \(\psi^2 = [\epsilon_p]\pi_{\EC}\) with \(\epsilon_p = \pm1\).
This situation is isomorphic to GLS.
In fact, \(\ECK \cong \conj{\ECK}\) 
implies \(j(\ECK) = j(\conj{\ECK}) = \conj{j(\ECK)}\),
so \(j(\ECK)\) is in \(\QQ\),
and \(\ECK\) is isomorphic to (or a quadratic twist of)
a curve defined over~\(\QQ\).
We note that in the case \(d = 1\),
we have \(r = \pm t_0\) in Proposition~\ref{prop:eigenvalue},
and the basis constructed in the proof of Proposition~\ref{prop:bitlength}
is (up to sign) the same as the basis of~\cite[Lemma~3]{GLS}.

While \(\EC'(\FF_{p^2})\) may have prime order,
\(\EC(\FF_{p^2})\) cannot:
the points fixed by \(\pi_0\) form a subgroup of order \(p + 1 - t_0\),
where \(t_0^2 - 2p = \trace{\EC}\)
(the complementary subgroup, where \(\pi_0\) has eigenvalue
\(-1\), has order \(p + 1 + t_0\)).
We see that the largest prime divisor of \(\#\EC(\FF_{p^2})\)
can be no larger than \(\Oh(p)\).
If we are in a position to apply the Fouque--Lercier--R\'eal--Valette
fault attack~\cite{Fouque--Lercier--Real--Valette}---for example,
if Montgomery ladders are used for scalar multiplication and
multiexponentiation---then we can solve DLP instances in \(\EC'(\FF_{p^2})\)
in \(\Oh(p^{1/2})\) group operations (in the worst case!).
While \(\Oh(p^{1/2})\) is still exponentially difficult,
it falls far short of the ideal \(\Oh(p)\)
for general curves over~\(\FF_{p^2}\).
GLS curves should therefore be avoided where the fault attack can be
put into practice.

\section{
    Higher degrees
}
\label{sec:further}

We conclude with some brief remarks on \(\QQ\)-curves of other small
degrees.
Hasegawa provides a universal curve for \(d = 7\) (and any \(\Delta\)) 
in~\cite[Theorem 2.2]{Hasegawa},
and our results for \(d = 2\) and \(d = 3\) carry over to \(d = 7\)
in an identical fashion,
though the endomorphism is slightly less efficient in this case
(its defining polynomials are sextic). 

For \(d = 5\), Hasegawa notes that it is impossible to give a universal
\(\QQ\)-curve for every discriminant \(\Delta\):
there exists a quadratic \(\QQ\)-curve of degree \(5\) 
over \(\QQ(\sqrt{\Delta})\)
if and only if 
\( \Legendre{5}{p_i} = 1 \)
for every prime  \(p_i \not= 5\) dividing \(\Delta\)
\cite[Proposition 2.3]{Hasegawa}.
But this is no problem when reducing modulo \(p\), if we are prepared to
give up the freedom of choosing \(\Delta\):
we could take
\( \Delta = -11\) for \(p \equiv 1 \pmod{4}\)
and \(\Delta = -1 \) for \(p \equiv 3 \pmod{4}\),
and then use the curves defined in~\cite[Table 6]{Hasegawa}.

Composite degree \(\QQ\)-curves (such as \(d = 6\) and \(10\))
promise more interesting results,
as do exceptional CM specializations of the universal curves;
we will return to these cases in future work.
Degrees greater than 10 yield less efficient endomorphisms,
and so are less interesting from a practical point of view.

\subsubsection{Acknowledgements}
The author thanks Fran\c{c}ois Morain and David Gruenewald
for their comments.

\end{document}